\newif\ifLTEX
\LTEXtrue

\documentclass[a4paper]{amsart}

\usepackage{mathtools,amssymb}
\usepackage[abbrev]{amsrefs}
\usepackage{xparse}

\ifLTEX
\usepackage{graphicx}
\else
\usepackage[dvipdfmx]{graphicx}
\fi
 
\usepackage{hyperref}
\usepackage{subcaption}
\usepackage{enumitem}
\usepackage{amsthm}
\usepackage{pdfpages}
\usepackage[all]{xy} 
\usepackage{float}
\usepackage{marginfix}
\newtheorem{thm}{Theorem}[section]
\newtheorem{lem}[thm]{Lemma}

\newtheorem{prop}[thm]{Proposition}

\theoremstyle{definition}

\newtheorem{rem}[thm]{Remark}

\newcommand{\R}{\mathbb{R}}

\newcommand{\Z}{\mathbb{Z}}

\newcommand{\M}{\mathrm{Mod}_{0,2n+2}}

\newcommand{\LM}{\mathrm{LMod}_{2n+2}}

\newcommand{\Mg}{\mathrm{Mod}_{g}}
\newcommand{\Mgb}{\mathrm{Mod}_{g}^1}

\newcommand{\SM}{\mathrm{SMod}_{g;k}}

\newcommand{\B}{\mathcal{B}}

\numberwithin{equation}{section}

\allowdisplaybreaks
\title[Positive factorization for balanced superelliptic rotation]{A positive factorization for the balanced superelliptic rotation}
\author[G.~Omori]{Genki Omori}
\address{
(Genki Omori)
Department of Mathematics, Faculty of Science and Technology, Tokyo University of Science, 2641 Yamazaki, Noda-shi, Chiba, 278-8510 Japan
}
\email{omori\_genki@ma.noda.tus.ac.jp}

\if0
\author[]{}
\address{
()
}
\email{}
\fi


\date{\today}

\begin{document}
\maketitle
\begin{abstract}
The balanced superelliptic rotation is a periodic map on an oriented closed surface of order $k\geq 3$. 
We give a positive factorization for the balanced superelliptic rotation. 
\end{abstract}

\section{Introduction}

%\subsection{Background}

Let $\Sigma _{g}$ be a connected closed oriented surface of genus $g\geq 0$. 
For a subset $A$ of $\Sigma _g$, the {\it mapping class group} $\mathrm{Mod}(\Sigma _{g}, A)$ of the pair $(\Sigma _{g}, A)$ is the group of isotopy classes of orientation-preserving self-diffeomorphisms on $\Sigma _{g}$ which preserve $A$ setwise. 
When $A$ is a set of distinct $n$ points, we denote the mapping class group by $\mathrm{Mod}_{g,n}$. 
We denote simply $\mathrm{Mod}_{g}=\mathrm{Mod}_{g,0}$. 
Dehn~\cite{Dehn} proved that $\mathrm{Mod}_{g}$ is generated by Dehn twists. 
As a well-known fact, a left-handed Dehn twist is a product of right-handed Dehn twists in $\Mg$. 
Hence every $f\in \Mg $ is expressed a product of right-handed Dehn twists. 
We call the product a \textit{positive factorization} for $f$. 

A self-diffeomorphism $\varphi $ on $\Sigma _g$ is a \textit{periodic map} if there exists an integer $k$ such that $\varphi ^k=1$. 
Positive factorizations of periodic maps were given by a lot of research, for instance, by Birman-Hilden~\cite{Birman-Hilden1}, Gurtas~\cite{Gurtas1, Gurtas2, Gurtas3}, Hirose~\cite{Hirose}, Ishizaka~\cite{Ishizaka}, Korkmaz~\cite{Korkmaz}, and Matsumoto~\cite{Matsumoto}. 
Birman and Hilden~\cite{Birman-Hilden1} gave a positive factorization for the hyperelliptic involution. 
In this paper, we give a positive factorization for a periodic map of order $k\geq 3$ which is a generalization of the hyperelliptic involution and is called the \textit{balanced superelliptic rotation}.  

For integers $n\geq 1$ and $k\geq 2$, we assume that $g=n(k-1)$. 
The \textit{balanced superelliptic rotation} $\zeta =\zeta _{g,k}$ is a periodic map on $\Sigma _g$ of order $k\geq 2$ with $2n+2$ fixed points $\widetilde{p}_1,\ \widetilde{p}_2,\ \dots ,\ \widetilde{p}_{2n+2}\in \Sigma _{g}$ as on the upper side in Figure~\ref{fig_bs_periodic_map} (for a precise definition, see Section~\ref{section_bscov}). 
When $k=2$, $\zeta =\zeta _{g,2}$ coincides with the hyperelliptic involution, and for $k\geq 3$, the balanced superelliptic rotation was introduced by Ghaswala and Winarski~\cite{Ghaswala-Winarski2}. 

\begin{figure}[h]
\includegraphics[scale=1.3]{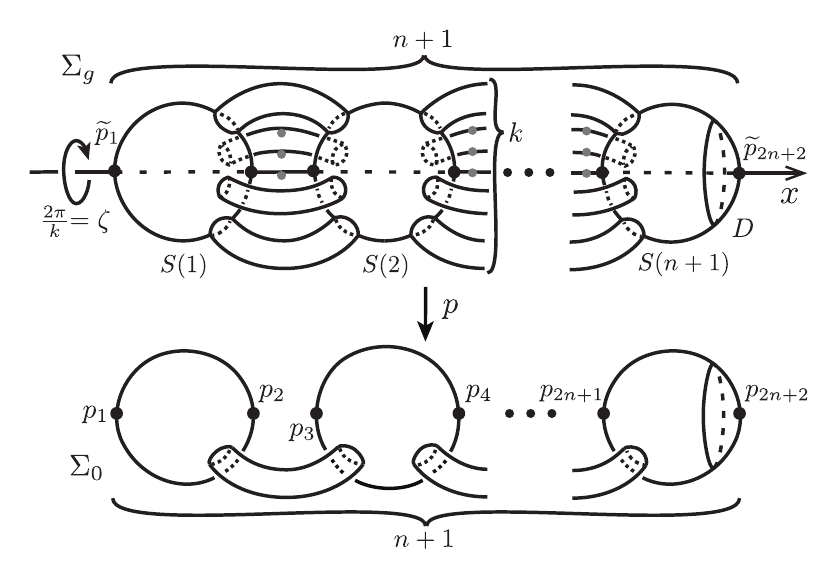}
\caption{The balanced superelliptic rotation $\zeta =\zeta _{g,k}$ on the balanced superelliptic covering map $p=p_{g,k}\colon \Sigma _g\to \Sigma _0$.}\label{fig_bs_periodic_map}
\end{figure}

Let $\widetilde{l}_i^l$ for $1\leq i\leq 2n+1$ and $1\leq l\leq k$ be a simple arc on $\Sigma _g$ which satisfies $\partial \widetilde{l}_i^l=\{ \widetilde{p}_{i}, \widetilde{p}_{i+1}\}$, $\zeta (\widetilde{l}_i^l)=\widetilde{l}_i^{l+1}$ for  $1\leq l\leq k-1$, and $\zeta (\widetilde{l}_i^k)=\widetilde{l}_i^{1}$ as on the upper side in Figure~\ref{fig_isotopy_surface_3-handles}. 
We consider a diffeomorphism of $\Sigma _g$ as in Figure~\ref{fig_isotopy_surface_3-handles} and identify $\Sigma _g$ with the surface as on the lower side in Figure~\ref{fig_isotopy_surface_3-handles}. 
Let $\gamma _i^l$ for $1\leq i\leq 2n+1$ and $1\leq l\leq k$ be a simple closed curve on $\Sigma _g$ which satisfies $\zeta (\gamma _i^l)=\gamma _i^{l+1}$ for $1\leq l\leq k-1$ as in Figure~\ref{fig_scc_c_il} and $\alpha _{i}^{l}$ for $1\leq i\leq n$ and $1\leq l\leq k-1$ a simple closed curve on $\Sigma _g$ which satisfies $\zeta (\alpha _i^l)=\alpha _i^{l+1}$ for $1\leq l\leq k-1$ as in Figure~\ref{fig_scc_a_il}. 
For a simple closed curve $\gamma $ on $\Sigma _g$, we denote by $t_\gamma $ the right-handed Dehn twist along $\gamma $. 
For $f,\ h\in \Mg $, the product $hf\in \Mg $ means that $f$ apply first and we abuse notation and denote a diffeomorphism and its isotopy class by the same symbol.  

\begin{figure}[h]
\includegraphics[scale=0.7]{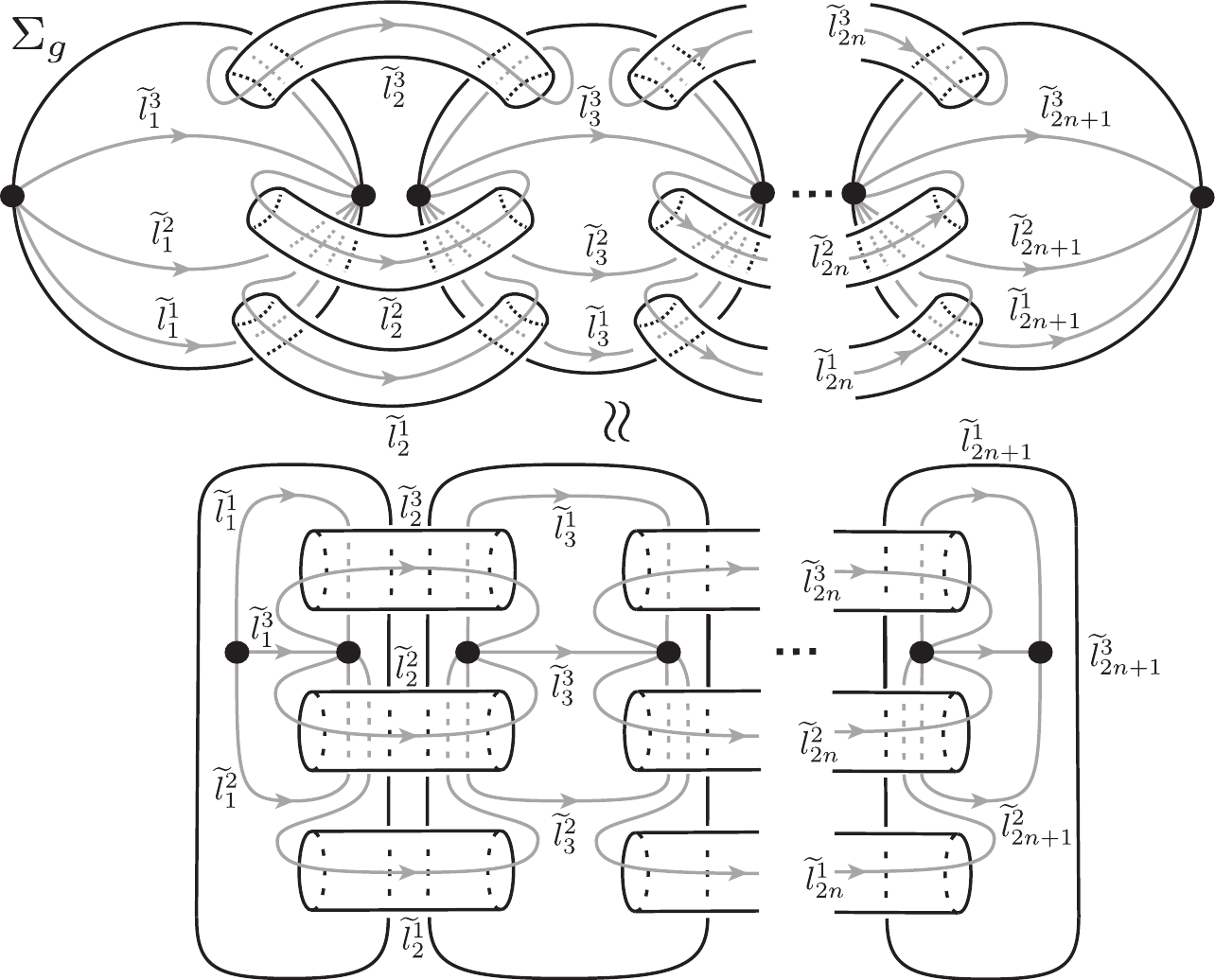}
\caption{A natural diffeomorphism of $\Sigma _g$ when $k=3$.}\label{fig_isotopy_surface_3-handles}
\end{figure}

\begin{figure}[h]
\includegraphics[scale=0.85]{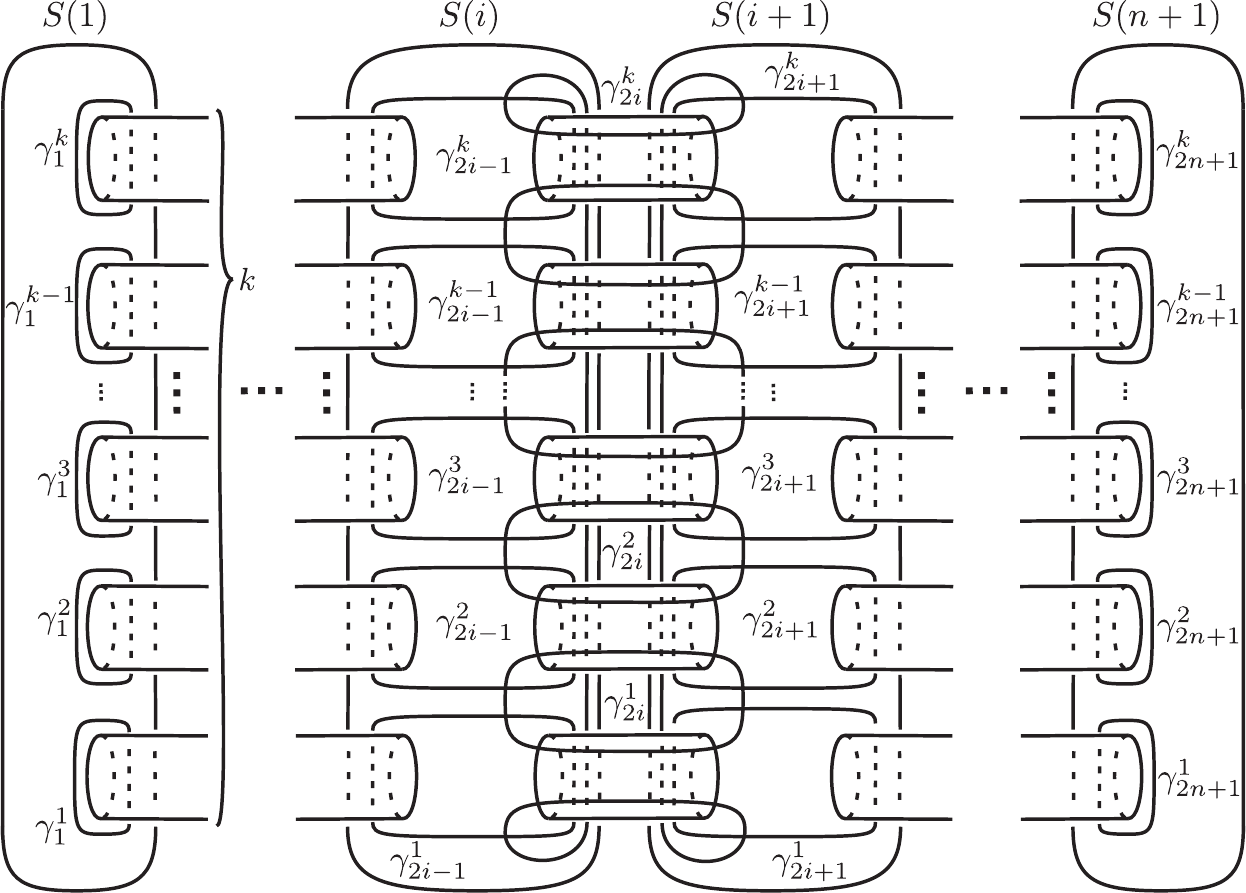}
\caption{Simple closed curves $\gamma _i^l$ on $\Sigma _g$ for $1\leq i\leq 2n+1$ and $1\leq l\leq k$.}\label{fig_scc_c_il}
\end{figure}

\begin{figure}[h]
\includegraphics[scale=0.83]{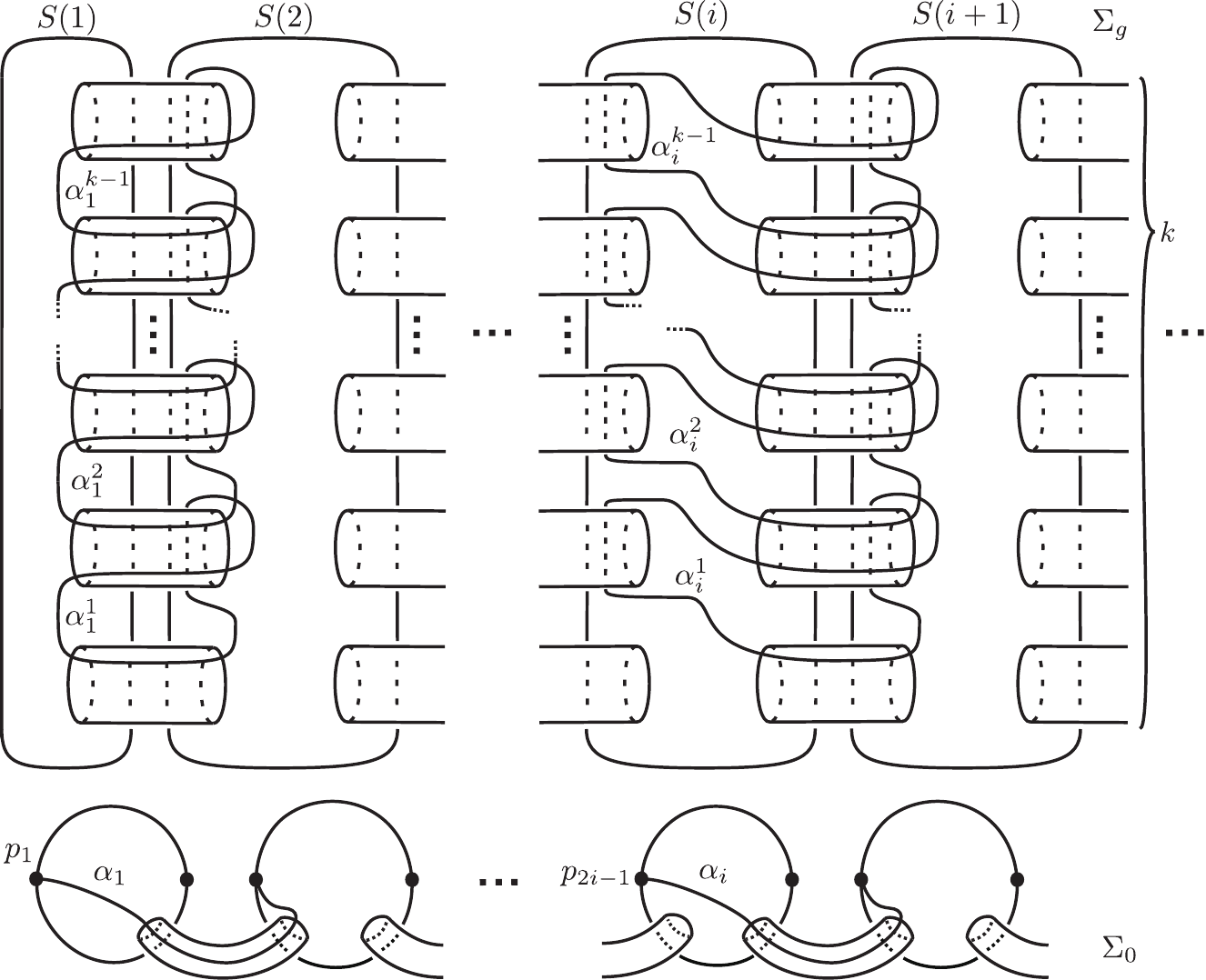}
\caption{Simple closed curves $\alpha _i^l$ on $\Sigma _g$ for $1\leq i\leq n$ and $1\leq l\leq k-1$.}\label{fig_scc_a_il}
\end{figure}

\newpage
Put
\begin{itemize}
\item $\widetilde{h}_{2i-1}=t_{\gamma _{2i-1}^1}t_{\gamma _{2i}^{1}}t_{\gamma _{2i-1}^2}t_{\gamma _{2i}^{2}}t_{\gamma _{2i-1}^3}\cdots t_{\gamma _{2i}^{k-1}}t_{\gamma _{2i-1}^{k}}$ for $1\leq i\leq n$, 
\item $\widetilde{a}_i=t_{\alpha _i^1}t_{\alpha _i^2}\cdots t_{\alpha _i^{k-1}}$\quad for $1\leq i\leq n$,
\item $\widetilde{t}_{2n+1}=t_{\gamma _{2n+1}^1}t_{\gamma _{2n+1}^2}\cdots t_{\gamma _{2n+1}^{k}}$. 
\end{itemize}
The main theorem in this paper is as follows. 

\begin{thm}\label{main_thm}
Let $\zeta =\zeta _{g,k}$ be the balanced superelliptic rotation on $\Sigma _g$ for $g=n(k-1)$ with $n\geq 1$ and $k\geq 3$. 
Then the relation
\[
\zeta =\widetilde{h}_1\widetilde{h}_3\cdots \widetilde{h}_{2n-1}\widetilde{t}_{2n+1}\widetilde{a}_n\cdots \widetilde{a}_2\widetilde{a}_1
\]
holds in $\Mg $. 
\end{thm}

We prove Theorem~\ref{main_thm} in Section~\ref{section_proof_main-thm}. 
As an application, in Section~\ref{section_application}, we observe topological properties of the Lefschetz fiblation over 2-sphere corresponding to the relation $(\widetilde{h}_1\widetilde{h}_3\cdots \widetilde{h}_{2n-1}\widetilde{t}_{2n+1}\widetilde{a}_n\cdots \widetilde{a}_2\widetilde{a}_1)^k=1$.

\section{Preliminaries}\label{Preliminaries}

\subsection{The balanced superelliptic covering space}\label{section_bscov}

In this section, we review the definitions of the balanced superelliptic covering space and the balanced superelliptic rotation from Section~2.1 in \cite{Hirose-Omori}. 
%Let $\Sigma _{g}$ be a connected closed oriented surface of genus $g\geq 0$. 
For integers $n\geq 1$ and $k\geq 2$ with $g=n(k-1)$, we describe the surface $\Sigma _{g}$ as follows. 
We take the unit 2-sphere $S^2=S(1)$ in $\R ^3$ and $n$ mutually disjoint parallel copies $S(2),\ S(3),\ \dots ,\ S(n+1)$ of $S(1)$ by translations along the x-axis such that 
\[
\max \bigl( S(i)\cap (\R \times \{ 0\}\times \{ 0\} )\bigr) <\min \bigl( S(i+1)\cap (\R \times \{ 0\}\times \{ 0\} )\bigr)
\]
for $1\leq i\leq n$ (see Figure~\ref{fig_bs_periodic_map}). 
Let $\zeta$ be the $(-\frac{2\pi }{k})$-rotation of $\R ^3$ on the $x$-axis. 
Then we remove $2k$ disjoint open disks in $S(i)$ for $2\leq i\leq n$ and $k$ disjoint open disks in $S(i)$ for $i\in \{ 1,\ n+1\}$ which are setwisely preserved by the action of $\zeta $, and connect $k$ boundary components of the punctured $S(i)$ and ones of the punctured $S(i+1)$ by $k$ annuli such that the union of the $k$ annuli is preserved by the action of $\zeta $ for each $1\leq i\leq n$ as in Figure~\ref{fig_bs_periodic_map}. 
Since the union of the punctured $S(1)\cup S(2)\cup \cdots \cup S(n+1)$ and the attached $n\times k$ annuli is diffeomorphic to $\Sigma _{g=n(k-1)}$, 
%where $g=n(k-1)$, 
we regard this union as $\Sigma _g$. 

By the construction above, the action of $\zeta $ on $\R ^3$ induces the action on $\Sigma _g$. 
We call $\zeta _{g,k}=\zeta |_{\Sigma _g}\colon \Sigma _g\to \Sigma _g$ the \textit{balanced superelliptic rotation} on $\Sigma _g$ and denote simply $\zeta _{g,k}=\zeta $. 
When $k=2$, $\zeta _{g,2}$ coincides with the hyperelliptic involution. 

Remark that the quotient space $\Sigma _g/\left< \zeta \right>$ is diffeomorphic to $\Sigma _0$ and the quotient map $p=p_{g,k}\colon \Sigma _g\to \Sigma _0$ is a branched covering map with $2n+2$ branch points in $\Sigma _0$. 
We call the branched covering map $p\colon \Sigma _g\to \Sigma _0$ the \textit{balanced superelliptic covering map}. 
Denote by $\widetilde{p}_1,\ \widetilde{p}_2,\ \dots ,\ \widetilde{p}_{2n+2}\in \Sigma _g$ the fixed points of $\zeta $ such that $\widetilde{p}_i<\widetilde{p}_{i+1}$ in $\R =\R \times \{ 0\} \times \{ 0\}$ for $1\leq i\leq 2n+1$, by $p_i\in \Sigma _0$ for $1\leq i\leq 2n+2$ the image of $\widetilde{p}_i$ by $p$ (i.e. $p_1,\ p_2,\ \dots ,\ p_{2n+2}\in \Sigma _0$ are branch points of $p$), and by $\B $ the set of the branch points $p_1,\ p_2,\ \dots ,\ p_{2n+2}$.

\subsection{The Birman-Hilden correspondence}\label{section_BH}

In this section, we review the Birman-Hilden correspondence~\cite{Birman-Hilden2}. 
For $g=n(k-1)\geq 1$, an orientation-preserving self-diffeomorphism $\varphi $ on $\Sigma _{g}$ is \textit{symmetric} for $\zeta =\zeta _{g,k}$ if $\varphi \left< \zeta \right> \varphi ^{-1}=\left< \zeta \right> $. 
The \textit{balanced superelliptic mapping class group} (or \textit{symmetric mapping class group}) $\SM $ is the subgroup of $\Mg $ which consists of elements represented by symmetric diffeomorphisms. 
$\mathrm{SMod}_{g;2}$ is called the \textit{hyperelliptic mapping class group}. 
Birman and Hilden~\cite{Birman-Hilden3} showed that $\SM $ coincides with the group of symmetric isotopy classes of symmetric diffeomorphisms on $\Sigma _g$. 

An orientation-preserving self-diffeomorphism $\varphi $ on $\Sigma _{0}$ is \textit{liftable} with respect to $p=p_{g,k}$ if there exists an orientation-preserving self-diffeomorphism $\widetilde{\varphi }$ on $\Sigma _{g}$ such that $p\circ \widetilde{\varphi }=\varphi \circ p$, namely, the following diagram commutes: 
\[
\xymatrix{
\Sigma _g \ar[r]^{\widetilde{\varphi }} \ar[d]_p &  \Sigma _{g} \ar[d]^p \\
\Sigma _{0}  \ar[r]_{\varphi } &\Sigma _{0}. \ar@{}[lu]|{\circlearrowright} 
}
\] 
Put $\B =\{ p_1,\ p_2,\ \dots ,\ p_{2n+2}\}$ and we regard $\M =\mathrm{Mod}(\Sigma _0, \B )$. 
The \textit{liftable mapping class group} $\mathrm{LMod}_{2n+2;k}$ is the subgroup of $\M $ which consists of elements represented by liftable diffomorphisms. 
For $k=2$, we have $\mathrm{LMod}_{2n+2;2}=\M $ by Birman and Hilden~\cite{Birman-Hilden1}. 
Since symmetric diffeomorphisms for $\zeta $ preserve $\widetilde{\B }=p^{-1}(\B )\subset \Sigma _g$, we have the natural homomorphism $\theta \colon \SM \to \mathrm{LMod}_{2n+2;k}$. 
By Birman and Hilden~\cite{Birman-Hilden2}, we have the following lemma. 

\begin{lem}\label{lem_BH}
For $n\geq 1$ and $k\geq 2$ with $g=n(k-1)$, we have the following exact sequence: 
\begin{eqnarray}\label{exact_BH}
1\longrightarrow \left< \zeta \right> \longrightarrow \SM \stackrel{\theta }{\longrightarrow }\mathrm{LMod}_{2n+2;k} \longrightarrow 1. 
\end{eqnarray}
\end{lem}

\subsection{Liftable elements for the balanced superelliptic covering map}\label{section_liftable-element}

Assume that $k\geq 3$ in Section~\ref{section_liftable-element}. 
In this section, we introduce some liftable diffeomorphisms on $\Sigma _0$ for $p=p_{g,k}$ from Section~2.2 in \cite{Hirose-Omori}. 
Let $l_i$ $(1\leq i\leq 2n+1)$ be an oriented simple arc on $\Sigma _0$ whose endpoints are $p_i$ and $p_{i+1}$ as in Figure~\ref{fig_path_l}. 
Put $L=l_1\cup l_2\cup \cdots \cup l_{2n+1}$. 
The isotopy class of a diffeomorphism $\varphi $ on $\Sigma _0$ relative to $\B $ is determined by the isotopy class of the image of $L$ by $\varphi $ relative to  $\B $. 
We identify $\Sigma _0$ with the surface on the lower side in Figure~\ref{fig_path_l} by some diffeomorphism of $\Sigma _0$. 

\begin{figure}[h]
\includegraphics[scale=1.4]{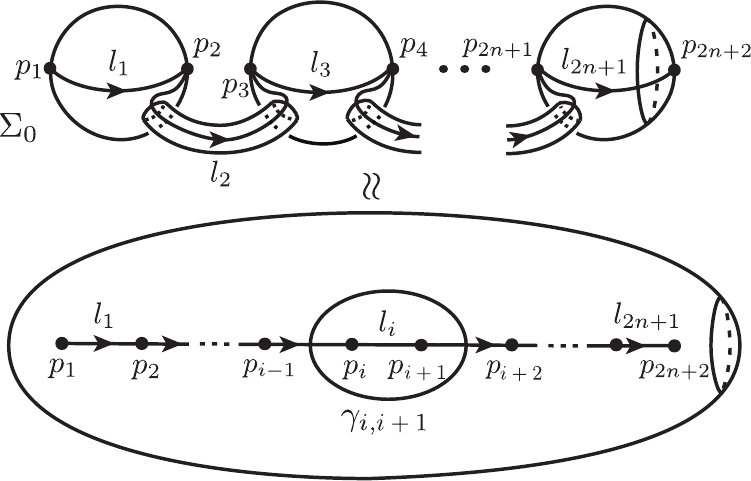}
\caption{A natural diffeomorphism of $\Sigma _0$, arcs $l_1,\ l_2,\ \dots ,\ l_{2n+1}$, and a simple closed curve $\gamma _{i,j}$ on $\Sigma _0$ for $1\leq i<j\leq 2n+2$.}\label{fig_path_l}
\end{figure}

Let $l$ be a simple arc on $\Sigma _0$ whose endpoints lie in $\B $. 
A regular neighborhood $\mathcal{N}$ of $l$ in $\Sigma _0$ is diffeomorphic to a 2-disk. 
Then the half-twist $\sigma [l]$ is a self-diffeomorphism on $\Sigma _0$ which is described as the result of anticlockwise half-rotation of $l$ in $\mathcal{N}$ as in Figure~\ref{fig_sigma_l}. 
$\sigma [l]$ is called the \textit{half-twist} along $l$. 
We define $\sigma _i=\sigma [l_i]$ for $1\leq i\leq 2n+1$. 
As a well-known result, $\M $ is generated by $\sigma _1$, $\sigma _2,\ \dots $, $\sigma _{2n+1}$ (see for instance Section~9.1.4 in \cite{Farb-Margalit}). 
Since $\M $ is naturally acts on $\B $, we have the surjective homomorphism 
\[
\Psi \colon \M \to S_{2n+2}
\]
given by $\Psi (\sigma _i)=(i\ i+1)$, where $S_{2n+2}$ is the symmetric group of degree $2n+2$. 

\begin{figure}[h]
\includegraphics[scale=1.1]{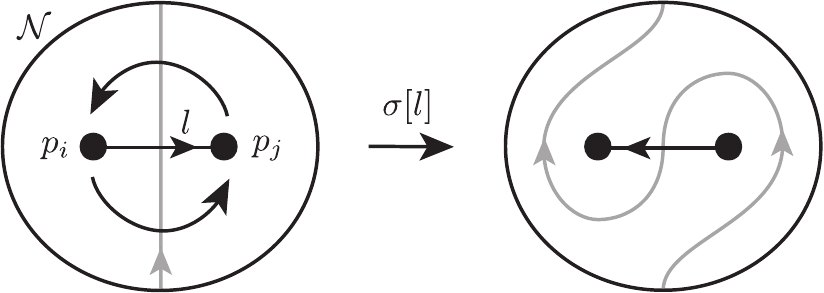}
\caption{The half-twist $\sigma [l]$ along $l$.}\label{fig_sigma_l}
\end{figure}

Put $\B _o=\{ p_1,\ p_3,\ \dots ,\ p_{2n+1}\}$ and $\B _e=\{ p_2,\ p_4,\ \dots ,\ p_{2n+2}\}$. 
An element $\sigma $ in $S_{2n+2}$ is \textit{parity-preserving} if $\sigma (\B _o)=\B _o$, and is \textit{parity-reversing} if $\sigma (\B _o)=\B _e$. 
An element $f$ in $\M $ is \textit{parity-preserving} (resp. \textit{parity-reversing}) if $\Psi (f)$ is \textit{parity-preserving} (resp. \textit{parity-reversing}). 
Let $W_{2n+2}$ be the subgroup of $S_{2n+2}$ which consists of parity-preserving or parity-reversing elements. 
Ghaswala and Winarski~\cite{Ghaswala-Winarski1} proved the following lemma. 
\begin{lem}[Lemma~3.6 in \cite{Ghaswala-Winarski1}]\label{lem_GW}
Let $\LM $ be the liftable mapping class group for the balanced superelliptic covering map $p_{g,k}$ for $n\geq 1$ and $k\geq3$ with $g=n(k-1)$. 
Then we have
\[
\LM =\Psi ^{-1}(W_{2n+2}).
\]
\end{lem}
Lemma~\ref{lem_GW} implies that a mapping class $f\in \M $ lifts with respect to $p_{g,k}$ if and only if $f$ is parity-preserving or parity-reversing (in particular, when $k\geq 3$, the liftability of a diffeomorphism on $\Sigma _0$ does not depend on $k$). 
Hence we omit ``$k$'' in the notation of the liftable mapping class groups for $k\geq 3$ (i.e. we express $\mathrm{LMod}_{2n+2;k}=\LM $ for $k\geq 3$).  
We will introduce some explicit liftable elements as follows. 
\\

\paragraph{\emph{The Dehn twist $t_{i}$}}

Let $\gamma _{i,i+1}$ for $1\leq i\leq 2n+1$ be a simple closed curve on $\Sigma _0-\B $ such that $\gamma _{i,i+1}$ surrounds the two points $p_i$ and $p_{i+1}$ as in Figure~\ref{fig_path_l}. 
Then we define $t_{i}=t_{\gamma _{i,i+1}}$ for $1\leq i\leq 2n+1$. 
Since $\Psi (t_{i})=1 \in W_{2n+2}$, we have $t_i\in \LM $ for $1\leq i\leq 2n+1$ by Lemma~\ref{lem_GW}. \\

\paragraph{\emph{The half-twist $a_i$}}

Let $\alpha _i$ for $1\leq i\leq n$ be a simple arc whose endpoints are $p_{2i-1}$ and $p_{2i+1}$ as in Figure~\ref{fig_arcs_a_ib_i}. 
Then we define $a_i=\sigma [\alpha _i]$ for $1\leq i\leq n$. 
Note that $a_i=\sigma _{2i}\sigma _{2i-1}\sigma _{2i}^{-1}$ for $1\leq i\leq n$. 
Since $\Psi (a_i)=(2i-1\ 2i+1)$ for $1\leq i\leq n$, the mapping class $a_i$ is parity-preserving and $a_i\in \LM$ by Lemma~\ref{lem_GW}. \\

\begin{figure}[h]
\includegraphics[scale=1.3]{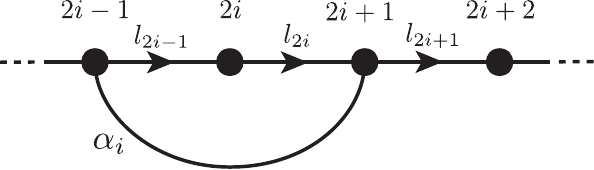}
\caption{An arc $\alpha _i$ for $1\leq i\leq n$ on $\Sigma _0$.}\label{fig_arcs_a_ib_i}
\end{figure}

\paragraph{\emph{The half-rotation $h_i$}}

Let $\mathcal{N}$ be a regular neighborhood of $l_i\cup l_{i+1}$ in $(\Sigma _0-\B )\cup \{ p_{i},\ p_{i+1},\ p_{i+2}\}$ for $1\leq i\leq 2n$. 
$\mathcal{N}$ is diffeomorphic to a 2-disk. 
Then we denote by $h_i$ the self-diffeomorphism on $\Sigma _0$ which is described as the result of anticlockwise half-rotation of $l_i\cup l_{i+1}$ in $\mathcal{N}$ as in Figure~\ref{fig_h_i}. 
Note that $h_i=\sigma _i\sigma _{i+1}\sigma _i$ for $1\leq i\leq 2n$. 
Since $\Psi (h_i)=(i\ i+2)$ for $1\leq i\leq 2n$, the mapping class $h_i$ is parity-preserving and $h_i\in \LM$ by Lemma~\ref{lem_GW}. 

\begin{figure}[h]
\includegraphics[scale=0.95]{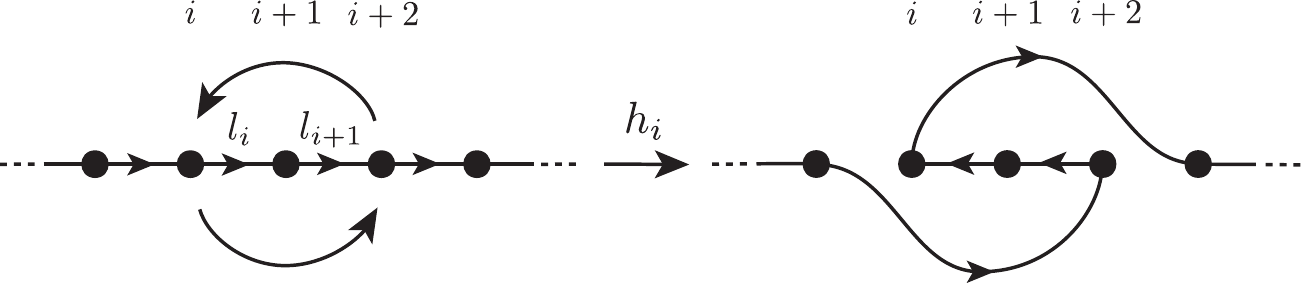}
\caption{The mapping class $h_i$ for $1\leq i\leq 2n$ on $\Sigma _0$.}\label{fig_h_i}
\end{figure}

\section{Proof of main theorem}\label{section_proof_main-thm}

In this section, we prove Theorem~\ref{main_thm}. 
Throughout this section, we assume that $n\geq 1$, $k\geq 3$, and $g=n(k-1)$.  
First, we have the following proposition. 

\begin{prop}\label{prop_LM_rel}
The relation
\[
h_1h_3\cdots h_{2n-1}t_{2n+1}a_n\cdots a_2a_1=1
\]
holds in $\LM $. 
\end{prop} 

\begin{proof}
Recall that the isotopy class of a diffeomorphism $\varphi $ on $\Sigma _0$ relative to $\B $ is determined by the isotopy class of the image of $L=l_1\cup l_2\cup \cdots \cup l_{2n+1}$ by $\varphi $ relative to  $\B $. 
By an argument in Figure~\ref{fig_proof_zeta}, we have $h_1h_3\cdots h_{2n-1}t_{2n+1}a_n\cdots a_2a_1(l_i)=l_i$ for $1\leq i\leq 2n+1$. 
Therefore, we have $h_1h_3\cdots h_{2n-1}t_{2n+1}a_n\cdots a_2a_1=1$ in $\LM$ and we have completed the proof of Proposition~\ref{prop_LM_rel}. 
\end{proof}

\begin{figure}[h]
\includegraphics[scale=1.0]{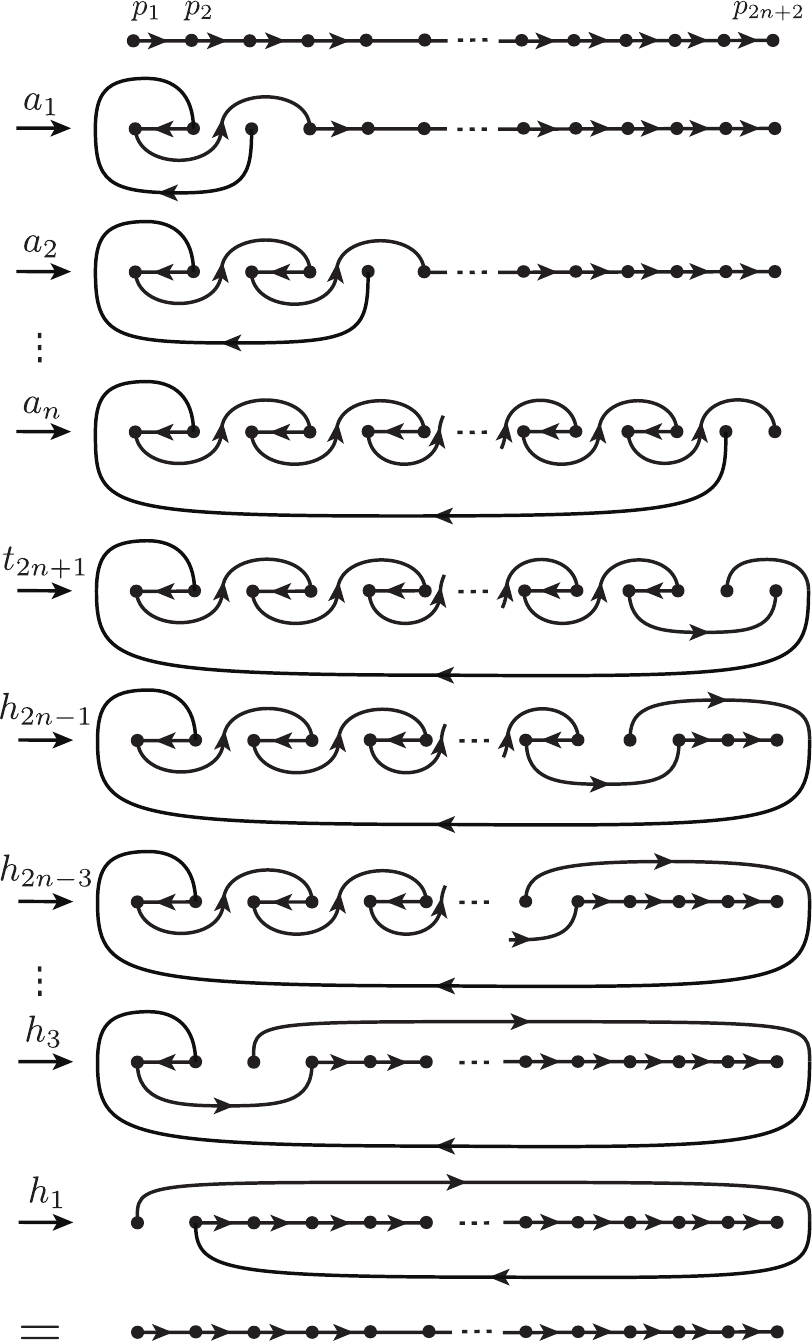}
\caption{The proof of $h_1h_3\cdots h_{2n-1}t_{2n+1}a_n\cdots a_2a_1(l_i)=l_i$ for $1\leq i\leq 2n+1$.}\label{fig_proof_zeta}
\end{figure}

Recall that simple closed curves $\gamma _i^l$ for $1\leq i\leq 2n+1$ and $1\leq l\leq k$, $\alpha _{i}^{l}$ for $1\leq i\leq n$ and $1\leq l\leq k-1$ on $\Sigma _g$ (see Figures~\ref{fig_scc_c_il} and \ref{fig_scc_a_il}) and mapping classes 
\begin{itemize}
\item $\widetilde{h}_{2i-1}=t_{\gamma _{2i-1}^1}t_{\gamma _{2i}^{1}}t_{\gamma _{2i-1}^2}t_{\gamma _{2i}^{2}}t_{\gamma _{2i-1}^3}\cdots t_{\gamma _{2i}^{k-1}}t_{\gamma _{2i-1}^{k}}$ for $1\leq i\leq n$, 
\item $\widetilde{a}_i=t_{\alpha _i^1}t_{\alpha _i^2}\cdots t_{\alpha _i^{k-1}}$\quad for $1\leq i\leq n$,
\item $\widetilde{t}_{2n+1}=t_{\gamma _{2n+1}^1}t_{\gamma _{2n+1}^2}\cdots t_{\gamma _{2n+1}^{k}}$. 
\end{itemize}

\begin{proof}[Proof of Theorem~\ref{main_thm}]
By Lemmas~6.3, 6.4, and 6.6 in ~\cite{Hirose-Omori}, the mapping classes $\widetilde{h}_{2i-1}$, $\widetilde{a}_i$, and $\widetilde{t}_{2n+1}$ are lifts of $h_{2i-1}$, $a_i$, and $t_{2n+1}$ with respect to $p$, respectively (namely, we have $\theta (\widetilde{h}_{2i-1})=h_{2i-1}$, $\theta (\widetilde{a}_i)=a_i$, and $\theta (\widetilde{t}_{2n+1})=t_{2n+1}$). 
Since $h_1h_3\cdots h_{2n-1}t_{2n+1}a_n\cdots a_2a_1=1$ in $\LM $ by Proposition~\ref{prop_LM_rel}, by the exact sequence~(\ref{exact_BH}) in Lemma~\ref{lem_BH}, there exists $0\leq l\leq k-1$ such that $\widetilde{h}_1\widetilde{h}_3\cdots \widetilde{h}_{2n-1}\widetilde{t}_{2n+1}\widetilde{a}_n\cdots \widetilde{a}_2\widetilde{a}_1=\zeta ^l$. 
Hence it is enough for completing the proof of Theorem~\ref{main_thm} to show that $\widetilde{h}_1\widetilde{h}_3\cdots \widetilde{h}_{2n-1}\widetilde{t}_{2n+1}\widetilde{a}_n\cdots \widetilde{a}_2\widetilde{a}_1(\gamma _{2n+1}^1)=\gamma _{2n+1}^2$. 

Since the support of $\widetilde{a}_i$ for $1\leq i\leq n-1$ is disjoint from $\gamma _{2n+1}^1$ and $\alpha _{n}^{l}$ for $2\leq l\leq k-1$ dose not intersects with $\gamma _{2n+1}^1$, we have 
\[
\widetilde{h}_1\widetilde{h}_3\cdots \widetilde{h}_{2n-1}\widetilde{t}_{2n+1}\widetilde{a}_n\cdots \widetilde{a}_2\widetilde{a}_1(\gamma _{2n+1}^1)=\widetilde{h}_1\widetilde{h}_3\cdots \widetilde{h}_{2n-1}\widetilde{t}_{2n+1}t_{\alpha _n^1}(\gamma _{2n+1}^1). 
\]
By an argument in Figure~\ref{fig_proof_zeta_prod}, we have $\widetilde{h}_{2n-1}\widetilde{t}_{2n+1}t_{\alpha _n^1}(\gamma _{2n+1}^1)=\gamma _{2n+1}^2$. 
Since the support of $\widetilde{h}_{2i-1}$ for $1\leq i\leq n-1$ is disjoint from $\gamma _{2n+1}^2$, we have $\widetilde{h}_1\widetilde{h}_3\cdots \widetilde{h}_{2n-1}\widetilde{t}_{2n+1}\widetilde{a}_n\cdots \widetilde{a}_2\widetilde{a}_1(\gamma _{2n+1}^1)=\gamma _{2n+1}^2$. 
Therefore we have completed the proof of Theorem~\ref{main_thm}. 
\end{proof}

\begin{figure}[h]
\includegraphics[scale=0.57]{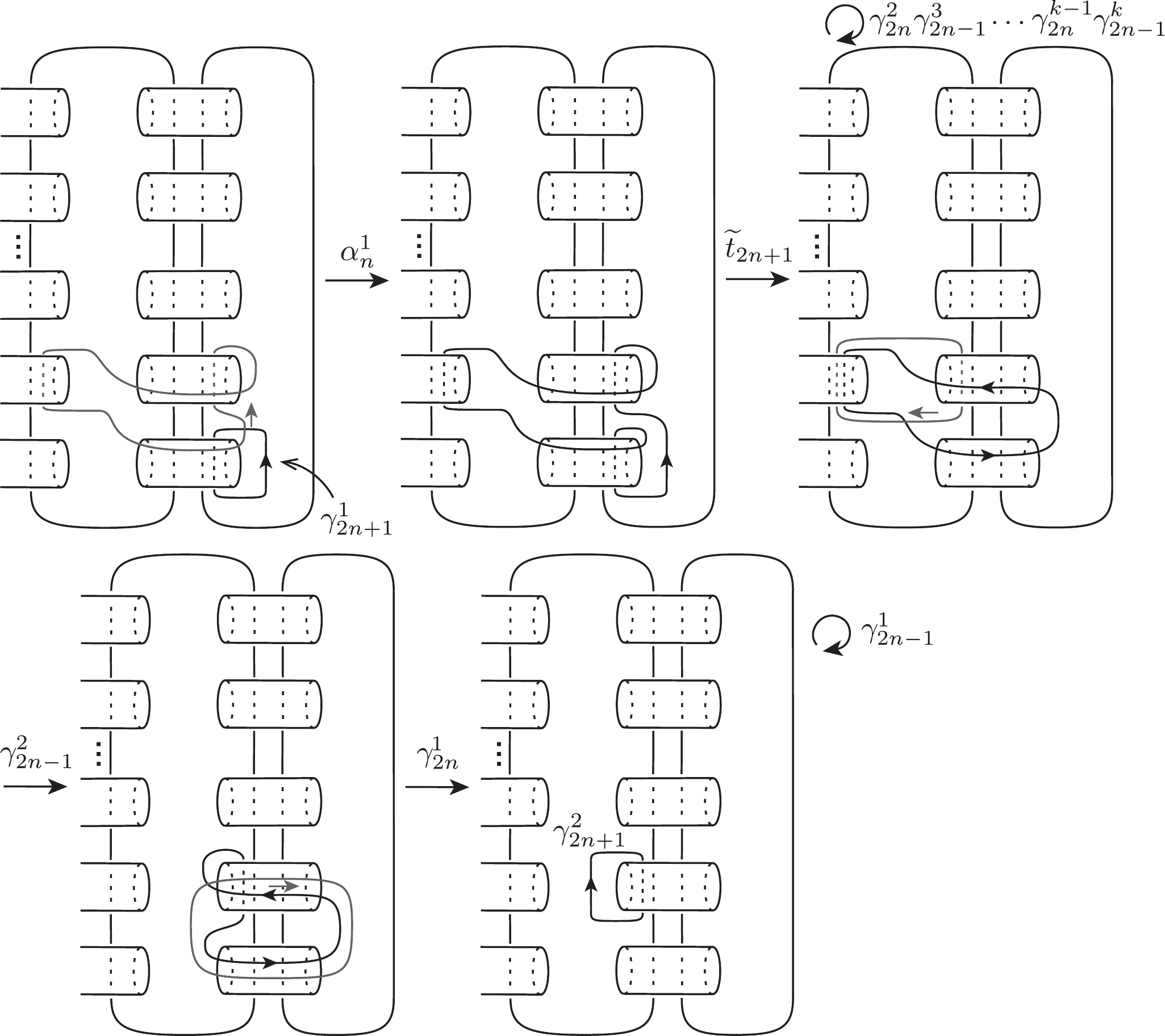}
\caption{The proof of $\widetilde{h}_{2n-1}\widetilde{t}_{2n+1}t_{\alpha _n^1}(\gamma _{2n+1}^1)=\gamma _{2n+1}^2$. We express the right-handed Dehn twist $t_\gamma $ along $\gamma $ by $\gamma $ in this figure. }\label{fig_proof_zeta_prod}
\end{figure}

\section{Examples of Lefschetz fibrations}\label{section_application}

Throughout this section, we assume that $g=n(k-1)$ for $n\geq 1$ and $k\geq 3$.  
In this section, we observe topological properties of the Lefschetz fibrations corresponding to the relation $(\widetilde{h}_1\widetilde{h}_3\cdots \widetilde{h}_{2n-1}\widetilde{t}_{2n+1}\widetilde{a}_n\cdots \widetilde{a}_2\widetilde{a}_1)^k=1$ in $\Mg $. 
First, we review the definition of the Lefschetz fibration and basic properties (for details, see \cite{Gompf-Stipsicz}). 

Let $X$ be a connected oriented closed smooth 4-manifold and $S^2$ an oriented 2-sphere. 
Then a smooth map $f\colon X\to S^2$ is a \textit{Lefschetz fibration} of genus $g$ if $f$ has finitely many critical points $x_1, \dots , x_m\in X$ with distinct images such that $f|_{X-\{ x_1, \dots , x_m\}}$ is $\Sigma _g$-bundle over $S^2-\{ f(x_1), \dots , f(x_m)\}$ and for each $1\leq i\leq m$, $f$ is locally expressed as $f(z_1, z_2)=z_1^2+z_2^2$ for some local complex coordinates around $x_i$ and $f(x_i)$ which are compatible with the orientations of $X$ and $S^2$.  
We take a base point $y\in S^2-\{ f(x_1), \dots , f(x_m)\}$ and identify the preimage $f^{-1}(y)$ with $\Sigma _g$ by some diffeomorphism. 
Each singular fiber $f^{-1}(f(x_i))$ is obtained by ``collapsing'' a single simple closed curve $c_i$ in the regular fiber $\Sigma _g$ to one point. 
We call %the preimage $f^{-1}(f(x_i))$ for $1\leq i\leq m$ a \textit{singular fiber} of $f$ and 
the simple close curve $c_i$ in $\Sigma _g$ a \textit{vanishing cycle} of $f$. 

For a simple smooth loop $\gamma $ on $S^2-\{ f(x_1), \dots , f(x_m)\}$ based at $y$, the pull-back bundle $\gamma ^\ast f$ is obtained from $\Sigma _g\times [0,1]$ by gluing $\Sigma _g\times \{ 0\}$ and $\Sigma _g\times \{ 1\}$ by an orientation-preserving diffeomorphism $\varphi $ on $\Sigma _g$. 
Then we have the well-defined antihomomorphism $\Phi \colon \pi _1(S^2-\{ f(x_1), \dots , f(x_m)\} ,y)\to \Mg$ defined by $\Phi (\gamma )=[\varphi ]\in \Mg $ and call $\Phi $ the \textit{monodromy representation} of $f$. 
Let $\gamma _1, \gamma _2, \dots , \gamma _m$ be simple loops on $S^2$ based at $y$ such that $\gamma _i$ surrounds $f(x_i)$ anticlockwisely and $\gamma _1\gamma _2\cdots \gamma _m=1$ in $\pi _1(S^2-\{ f(x_1), \dots , f(x_m)\} ,y)$. 
As a well-known fact, we have $\Phi (\gamma _i)=t_{c_i}$ for $1\leq i\leq m$. 
Hence we have a positive relation  
\[
t_{c_m}\cdots t_{c_2}t_{c_1}=1 \text{ in }\Mg .
\]
Conversely, for a positive relator $t_{c_m}\cdots t_{c_2}t_{c_1}$ in $\Mg $, we can construct a Lefschetz fiblation over $S^2$ whose vanishing cycles are $c_1, c_2, \dots , c_m$. 
By Kas~\cite{Kas} and Matsumoto~\cite{Matsumoto}, if $g\geq 2$, the isomorphism class of Lefschetz fibration of genus $g$ is determined by a corresponding positive relator among Dehn twist up to \textit{simultaneous conjugations}
\[
t_{c_m}\cdots t_{c_2}t_{c_1}\sim t_{\varphi (c_m)}\cdots t_{\varphi (c_2)}t_{\varphi (c_1)}\text{ for }\varphi \in \Mg 
\]
and elementary transformations
\begin{align*}
&t_{c_m}\cdots t_{c_{i+2}}t_{c_{i+1}}t_{c_{i}}t_{c_{i-1}}\cdots t_{c_1}\sim t_{c_m}\cdots t_{c_{i+2}}t_{t_{c_{i+1}}(c_{i})}t_{c_{i+1}}t_{c_{i-1}}\cdots t_{c_1},\\
&t_{c_m}\cdots t_{c_{i+2}}t_{c_{i+1}}t_{c_{i}}t_{c_{i-1}}\cdots t_{c_1}\sim t_{c_m}\cdots t_{c_{i+2}}t_{c_{i}}t_{t_{c_{i}}^{-1}(c_{i+1})}t_{c_{i-1}}\cdots t_{c_1}.
\end{align*}

Recall that $\zeta =\widetilde{h}_1\widetilde{h}_3\cdots \widetilde{h}_{2n-1}\widetilde{t}_{2n+1}\widetilde{a}_n\cdots \widetilde{a}_2\widetilde{a}_1$ in $\Mg$ by Theorem~\ref{main_thm} and the product is expressed by right-handed Dehn twists.  
Since $\zeta ^k=1$, we have $(\widetilde{h}_1\widetilde{h}_3\cdots \widetilde{h}_{2n-1}\widetilde{t}_{2n+1}\widetilde{a}_n\cdots \widetilde{a}_2\widetilde{a}_1)^k=1$ in $\Mg $. 
Let $f_{g,k}\colon X_{g,k}\to S^2$ be the Lefschetz fibration of genus $g=n(k-1)$ corresponding to the positive relator $(\widetilde{h}_1\widetilde{h}_3\cdots \widetilde{h}_{2n-1}\widetilde{t}_{2n+1}\widetilde{a}_n\cdots \widetilde{a}_2\widetilde{a}_1)^k$. 

For a Lefschetz fibration $f\colon X\to S^2$ of genus $g$, we denote by $\chi (X)$ the Euler characteristic of $X$. 
Then, we can see that $\chi (X)=-4(g-1)+s$, where $s$ is the number of the singular fibers of $f$. 
Thus we have the following proposition. 

\begin{prop}\label{bsLF_Euler}
For $n\geq 1$ and $k\geq 3$ with $g=n(k-1)$, we have 
\[
\chi (X_{g,k})=(3n+1)k^2-6nk+4n+4.
\] 
\end{prop}

\subsection{Simply connectedness of $X_{g,k}$}

In this subsection, we prove that $X_{g,k}$ is simply connected. 
For a Lefschetz fibration $f\colon X\to S^2$ of genus $g$, a map $s\colon S^2\to X$ is a \textit{$(-1)$-section} of $f$ if $f\circ s=\mathrm{id}_{S^2}$ and self-intersection number of $[s(S^2)]$ in $H_2(X;\Z )$ is $-1$. 
Let $\Mgb $ be the group of isotopy classes of the self-diffeomorphisms on $\Sigma _g$ fixing a disk $D$ in $\Sigma _g$ pointwise. 
Then we have the natural surjective homomorphism $\mathcal{F}\colon \Mgb \to \Mg$ which is called the \textit{forgetful map}. 
As a well-known fact, for a Lefschetz fibration $f\colon X\to S^2$ corresponding to a positive relator $t_{c_m}\cdots t_{c_2}t_{c_1}$ in $\Mg $, if the relation $t_{c_m}\cdots t_{c_2}t_{c_1}=1$ in $\Mg $ lifts to the relation $t_{c_m}\cdots t_{c_2}t_{c_1}=t_{\partial D}$ in $\Mgb $ with respect to $\mathcal{F}$, then $f$ has a $(-1)$-section. 
By using this fact, we have the following lemma. 

\begin{lem}\label{lem_section}
For $n\geq 1$ and $k\geq 3$ with $g=n(k-1)$, the Lefschetz fibration $f_{g,k}\colon X_{g,k}\to S^2$ has a $(-1)$-section. 
\end{lem}

\begin{proof}
We regard the disk $D$ as a small disk neighborhood of $\widetilde{p}_{2n+2}\in \Sigma _g$ which is preserved by $\zeta $ (for the definition of $\widetilde{p}_{2n+2}$, see Section~\ref{section_bscov}). 
Let $\zeta ^\prime $ be a self-diffeomorphism on $\Sigma _g$ which is described as a result of a $(-\frac{2\pi }{k})$-rotation of $\Sigma _g-D$ fixing the disk $D$ pointwise as in Figure~\ref{fig_lift_t_partial-d}. 
By the definition, we have $\mathcal{F}(\zeta ^\prime)=\zeta \in \Mg $ and ${\zeta ^\prime }^k=t_{\partial D}\in \Mgb $. 
The mapping classes $\widetilde{h}_{2i-1}$, $\widetilde{a}_i$, and $\widetilde{t}_{2n+1}$ are naturally regarded as elements in $\Mgb $ and we can show that $\zeta ^\prime =\widetilde{h}_1\widetilde{h}_3\cdots \widetilde{h}_{2n-1}\widetilde{t}_{2n+1}\widetilde{a}_n\cdots \widetilde{a}_2\widetilde{a}_1$ in $\Mgb $ by a similar argument as in Figure~\ref{fig_proof_zeta_prod} (see the action of the product on the arc $(\widetilde{l}_{2n+1}^1\cup \widetilde{l}_{2n+1}^2)-\mathrm{int}D$). 
Thus we have 
\[
(\widetilde{h}_1\widetilde{h}_3\cdots \widetilde{h}_{2n-1}\widetilde{t}_{2n+1}\widetilde{a}_n\cdots \widetilde{a}_2\widetilde{a}_1)^k=(\zeta ^\prime )^k=t_{\partial D}\in \Mgb 
\]
and the Lefschetz fibration $f_{g,k}$ has a $(-1)$-section. 
We have completed the proof of Lemma~\ref{lem_section}. 
\end{proof}

For a Lefschetz fibration with a section, the fundamental group of the total space is calculated by the following lemma (precisely see \cite{Gompf-Stipsicz}). 

\begin{figure}[h]
\includegraphics[scale=1.3]{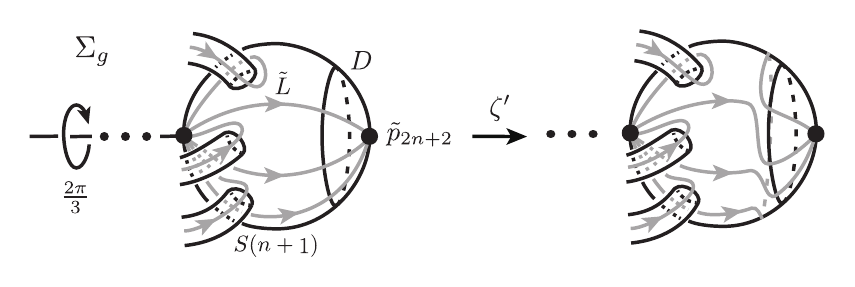}
\caption{The diffeomorphism $\zeta ^\prime $ on $\Sigma _g$ for $k=3$.}\label{fig_lift_t_partial-d}
\end{figure}

\begin{lem}\label{LF_pi1}
Let $f\colon X\to S^2$ be a Lefschetz fibration of genus $g$ corresponding to a positive relator $t_{c_m}\cdots t_{c_2}t_{c_1}$ in $\Mg $. 
Then, if $f$ has a section, then $\pi _1(X)$ is isomorphic to the quotient of $\pi _1(\Sigma _g)$ by the normal closure of $\{ c_1, c_2, \dots ,c_m\}$ in $\pi _1(X)$. 
\end{lem}

By using Lemma~\ref{LF_pi1}, we have the following proposition. 

\begin{prop}\label{bsLF_pi1}
$\pi _1(X_{g,k})=\{ 1\}$, namely, $X_{g,k}$ is simply connected. 
\end{prop}

\begin{proof}
First, we identify $\Sigma _g$ with the surface as on the right-hand side in Figure~\ref{fig_surface-homeo2} by a natural diffeomorphism. 
We take a base point $y\in \Sigma _g$ and generators $c_{i,l}$ for $1\leq i\leq 2n$ and $1\leq l\leq k-1$ of $\pi _1(\Sigma _g, y)$ as in Figure~\ref{fig_pi1_gen}. 
Let $N$ be the normal closure of $\{ \gamma _{2i-1}^l\mid 1\leq i\leq n+1,\ 1\leq l\leq k\} \cup \{ \gamma _{2i}^l\mid 1\leq i\leq n,\ 1\leq l\leq k-1\} \cup \{ \alpha _{i}^l\mid 1\leq i\leq n,\ 1\leq l\leq k-1\}$ in $\pi _1(\Sigma _g, y)$. 
By Lemma~\ref{LF_pi1}, $\pi _1(X_{g,k})$ is isomorphic to $\pi _1(\Sigma _g, y)/N$. 
Hence, it is enough for completing the proof of Proposition~\ref{bsLF_pi1} to prove that $c_{i,l}\in N$ for $1\leq i\leq 2n$ and $1\leq l\leq k-1$. 

As elements in $\pi _1(\Sigma _g, y)$, we have $\gamma _{2i}^l=c_{2i,l}$, $\gamma _{2i-1}^1=c_{2i-1,1}$, and $\gamma _{2i-1}^{k}=c_{2i-1,k-1}$ for $1\leq i\leq n$ and $1\leq l\leq k-1$. 
Hence these elements lie in $N$. 
Denote $\delta _{1,l}=c_{1,l}^{-1}(c_{2n,l-1}\cdots c_{4,l-1}c_{2,l-1})c_{1,l-1}(c_{2,l-1}^{-1}c_{4,l-1}^{-1}\cdots c_{2n,l-1}^{-1})$ for $2\leq l\leq k-2$ and 
\[
\delta _{i,l}=c_{2i-1,l}^{-1}\delta _{i-1,l}(c_{2n,l-1}\cdots c_{2i+2,l-1}c_{2i,l-1})c_{2i-1,l-1}(c_{2i,l-1}^{-1}c_{2i+2,l-1}^{-1}\cdots c_{2n,l-1}^{-1})
\]
for $2\leq i\leq n$ and $2\leq l\leq k-2$. 
Remark that $\delta _{i,l}$ for $1\leq i\leq n$ and $2\leq l\leq k-2$ is represented by an oriented simple loop on $\Sigma _g$ based at $y$ as in Figure~\ref{fig_loop_delta}. 
We can see that $\gamma _{1,l}=\delta _{1,l}$ and $\gamma _{2i-1,l}=\delta _{i-1,l}\delta _{i,l}^{-1}$ for $2\leq i\leq n$ and $2\leq l\leq k-2$. 
Thus, we show that 
\[
c_{1,2},\ c_{1,3}, \dots ,\ c_{1,k-2},\ c_{3,2},\ c_{3,3}, \dots ,\ c_{3,k-2}, \dots ,\ c_{2n-1,2},\ c_{2n-1,3}, \dots ,\ c_{2n-1,k-2}\in N 
\]
inductively. 
Therefore, all $c_{i,l}$'s lie in $N$ and we have completed the proof of Proposition~\ref{bsLF_pi1}. 
\if0
Since we have 
\begin{enumerate}
\item $\gamma _{1}^l=c_{1,l}^{-1}(c_{2n,l-1}\cdots c_{4,l-1}c_{2,l-1})c_{1,l-1}(c_{2,l-1}^{-1}c_{4,l-1}^{-1}\cdots c_{2n,l-1}^{-1})$ for $2\leq i\leq k-2$,  
\item $\gamma _{3}^l=\gamma _{1}^l\left( (c_{2n,l-1}\cdots c_{6,l-1}c_{4,l-1})c_{3,l-1}^{-1}(c_{4,l-1}^{-1}c_{6,l-1}^{-1}\cdots c_{2n,l-1}^{-1})(\gamma _{1}^l)^{-1}c_{3,l}\right) $ for $2\leq i\leq k-2$, 
\item $\gamma _{5}^l=\left( c_{3,l}^{-1}\gamma _{1}^l(c_{2n,l-1}\cdots c_{6,l-1}c_{4,l-1})c_{3,l-1}(c_{4,l-1}^{-1}c_{6,l-1}^{-1}\cdots c_{2n,l-1}^{-1})\right) $\\ $\cdot  (c_{2n,l-1}\cdots c_{8,l-1}c_{6,l-1})c_{5,l-1}^{-1}(c_{6,l-1}^{-1}c_{8,l-1}^{-1}\cdots c_{2n,l-1}^{-1})$\\ $\cdot (c_{2n,l-1}\cdots c_{6,l-1}c_{4,l-1})c_{3,l-1}^{-1}(c_{4,l-1}^{-1}c_{6,l-1}^{-1}\cdots c_{2n,l-1}^{-1})(\gamma _{1}^l)^{-1}c_{3,l}\cdot c_{5,l}$ for $2\leq i\leq k-2$, 
\end{enumerate}
in $\pi _1(\Sigma _g, y)$, $c_{1,l}\in N$ for $2\leq i\leq k-2$. 
\fi
\end{proof}

\begin{figure}[h]
\includegraphics[scale=0.60]{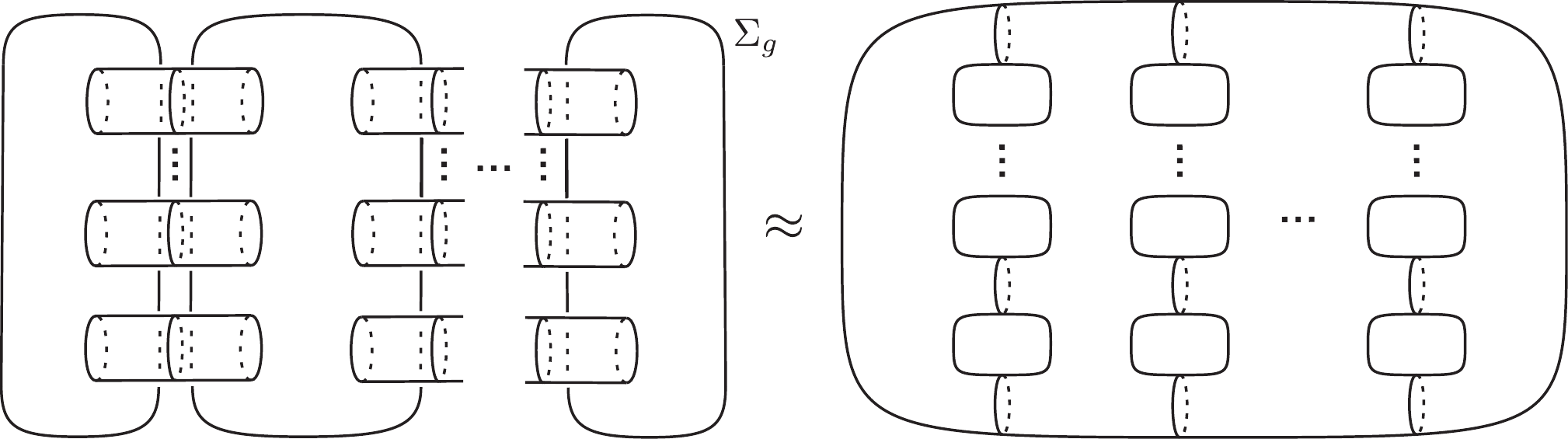}
\caption{A natural diffeomorphism of $\Sigma _g$.}\label{fig_surface-homeo2}
\end{figure}

\begin{figure}[h]
\includegraphics[scale=0.9]{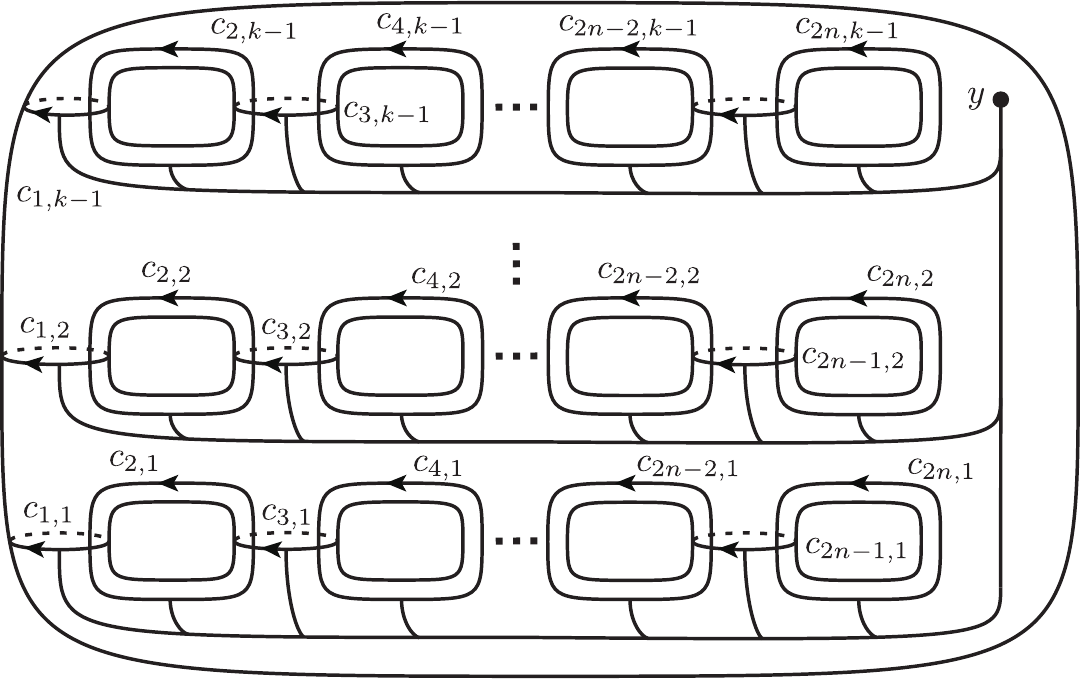}
\caption{A loop $c_{i,l}$ for $1\leq i\leq 2n$ and $1\leq l\leq k-1$ on $\Sigma _g$ based at $y$.}\label{fig_pi1_gen}
\end{figure}

\begin{figure}[h]
\includegraphics[scale=0.75]{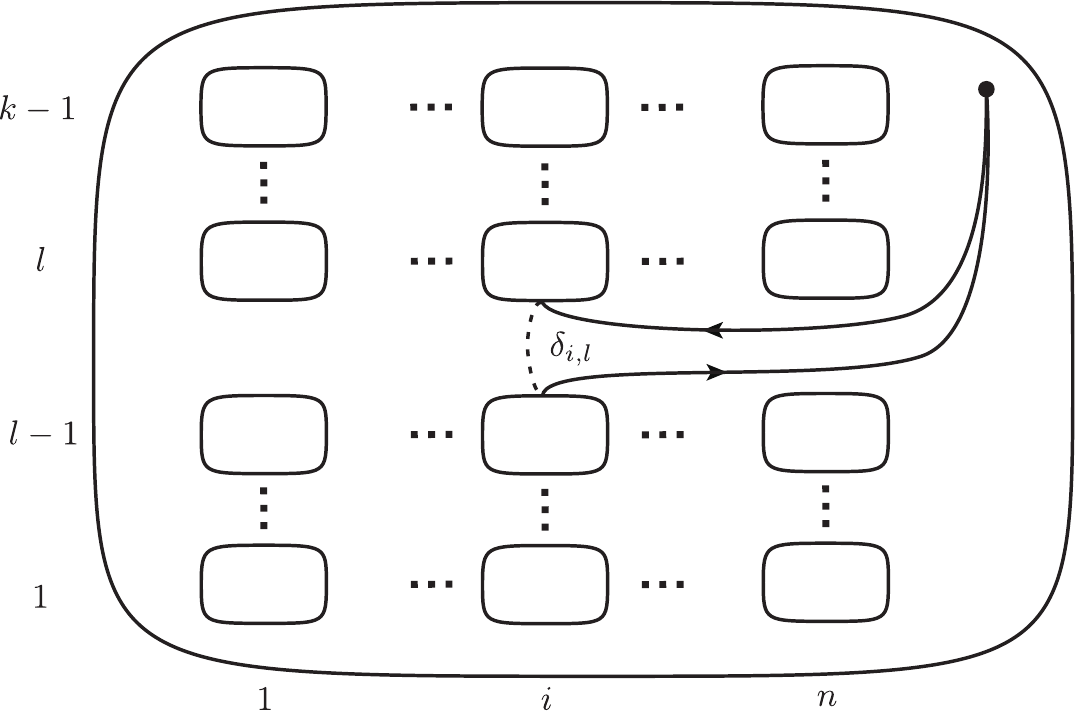}
\caption{A loop $\delta _{i,l}$ for $1\leq i\leq n$ and $2\leq l\leq k-2$ on $\Sigma _g$ based at $y$.}\label{fig_loop_delta}
\end{figure}

\subsection{Spinness of $X_{g,k}$}

In this subsection, we prove the following proposition. 

\begin{prop}\label{prop_nonspin}
For $n\geq 1$ and $k\geq 3$ with $g=n(k-1)$, $X_{g,k}$ has no spin structure. 
\end{prop}

To prove this proposition, we use a necessary and sufficient condition for the spinness of a Lefschetz fibration over $S^2$ or a 2-disk $D^2$ by Stipsicz~\cite{Stipsicz}. 
A map $q\colon H_{1}(\Sigma _g; \Z _2)\to \Z _2$ is a \emph{quadratic form} if $q(c+c^\prime )=q(c)+q(c^\prime )+c\cdot c^\prime $ for any $c,\ c^\prime \in H_{1}(\Sigma _g; \Z _2)$, where $\cdot $ is the intersection form. %on $H_{1}(\Sigma _g; \Z _2)$. 
For the case of Lefschetz fibrations over $D^2$, we have the following proposition. 

\begin{prop}[\cite{Stipsicz}]\label{prop_Stipsicz1}
Let $f\colon X\to D^2$ be a Lefschetz fibration of genus $g$ with a regular fiber $\Sigma _g=f^{-1}(y)$ and vanishing cycles $c_1, c_2, \dots , c_m\subset \Sigma _g$. 
Then $X$ admits a spin structure if and only if there exists a quadratic form $q \colon H_1(\Sigma _g;\Z _2)\to \Z _2$ such that $q(c_i)=1\in \Z _2=\{ 0,1\}$ for each $1\leq i\leq m$. 
\end{prop}

Proposition~\ref{prop_Stipsicz1} implies that a Lefschetz fibration with a separating vanishing cycle does not admit a spin structure. 
Let $f\colon X\to S^2$ be a Lefschetz fibration of genus $g$ with a regular fiber $\Sigma _g=f^{-1}(y)$. 
A homology class $\gamma \in H_2(X; \Z )$ is a \emph{dual} of $\Sigma _g$ if $[\Sigma _g]\cdot \gamma =1\in \Z$. 
We denote $\nu (\Sigma _g)=f^{-1}(U)$ for some small disk neighborhood $U$ of $y$ in $S^2$. 
Then the restriction $f|_{X-\nu (\Sigma _g)}\colon X-\nu (\Sigma _g)\to S^2-U\approx D^2$ is a Lefschetz fibration over $D^2$. 
For the case of Lefschetz fibrations over $S^2$, we have the following proposition. 

\begin{prop}[\cite{Stipsicz}]\label{prop_Stipsicz2}
Let $f\colon X\to S^2$ be a Lefschetz fibration of genus $g$ with a regular fiber $\Sigma _g=f^{-1}(y)$. 
Then $X$ admits a spin structure if and only if $X-\nu (\Sigma _g)$ is spin and $\gamma \cdot \gamma \equiv 0$ mod $2$ for some dual $\gamma \in H_2(X;\Z )$ of $\Sigma _g$. 
\end{prop}

Let $\Sigma _g=f_{g,k}^{-1}(y)$ be a regular fiber of $f_{g,k} \colon X_{g,k} \to S^2$. 
Recall that the vanishing cycles of $f_{g,k}$ are  $\gamma _{2i-1}^l$ $(1\leq i\leq n+1,\ 1\leq l\leq k)$, $\gamma _{2i}^l$ $(1\leq i\leq n,\ 1\leq l\leq k-1)$, and $\alpha _{i}^l$ $(1\leq i\leq n,\ 1\leq l\leq k-1)$ (see also Figures~\ref{fig_scc_c_il} and \ref{fig_scc_a_il}). 
From here, we often express a vanishing cycle and its homology class by the same symbol. 
The next lemma is the case that $k\geq 3$ is odd or $n\geq 2$ is even for Proposition~\ref{prop_nonspin}. %the Lefschetz fibration $f_{g,k} \colon X_{g,k} \to S^2$, we have the following lemma. 

\begin{lem}\label{lem_nonspin1}
If $k\geq 3$ is odd or $n\geq 2$ is even, then $X_{g,k}$ has no spin structure. 
\end{lem}

\begin{proof}
Let $q \colon H_1(\Sigma _g;\Z _2)\to \Z _2$ be a quadratic form such that $q(\gamma )=1\in \Z _2$ for each vanishing cycle $\gamma $ of $f_{g,k} \colon X_{g,k} \to S^2$. 
By Propositions~\ref{prop_Stipsicz1} and \ref{prop_Stipsicz2}, it is enough to prove that there is no such a quadratic form for $k\geq 3$ is odd or $n\geq 2$ is even. 

First, we assume that $k\geq 3$ is odd. 
As a $\Z _2$-homology class, we have $\gamma _1^1+\gamma _1^2+\cdots +\gamma _1^{k}=0 \in H_1(\Sigma _g;\Z _2)$. 
Since $\gamma _1^l\cdot \gamma _1^{l^\prime }=0\in \Z _2$ for $1\leq l, l^\prime \leq k$, we have 
\begin{eqnarray*}
0&=&q(\gamma _1^1+\gamma _1^2+\cdots +\gamma _1^{k})\\
&=&q(\gamma _1^1)+q(\gamma _1^2)+\cdots + q(\gamma _1^{k})\\
&=&k\\
&=&1\quad \text{in }\Z_2.
\end{eqnarray*}
This is a contradiction and we see that there is not a quadratic form $q \colon H_1(\Sigma _g;\Z _2)\to \Z _2$ such that $q(\gamma )=1\in \Z _2$ for each vanishing cycle $\gamma $ of $f_{g,k} \colon X_{g,k} \to S^2$ when $k\geq 3$ is odd. 

Next, we assume that $n\geq 2$ is even. 
As a $\Z _2$-homology class, we have $\gamma _1^1+\gamma _3^1+\cdots +\gamma _{2n+1}^{1}=0 \in H_1(\Sigma _g;\Z _2)$. 
Since $\gamma _{2i-1}^1\cdot \gamma _{2i^\prime -1}^{1}=0\in \Z _2$ for $1\leq i, i^\prime \leq n+1$, we have 
\begin{eqnarray*}
0&=&q(\gamma _1^1+\gamma _3^1+\cdots +\gamma _{2n+1}^{1})\\
&=&q(\gamma _1^1)+q(\gamma _3^1)+\cdots + q(\gamma _{2n+1}^{1})\\
&=&n+1\\
&=&1\quad \text{in }\Z_2.
\end{eqnarray*}
This is a contradiction and we see that there is not a quadratic form $q \colon H_1(\Sigma _g;\Z _2)\to \Z _2$ such that $q(\gamma )=1\in \Z _2$ for each vanishing cycle $\gamma $ of $f_{g,k} \colon X_{g,k} \to S^2$ when $n\geq 2$ is even. 
We have completed the proof of Lemma~\ref{lem_nonspin1}. 

\end{proof}

Proposition~\ref{prop_nonspin} is completed by Lemma~\ref{lem_nonspin1} and the next lemma. %is the case that $k\geq 3$ is odd or $n\geq 2$ is even of . 

\begin{lem}\label{lem_nonspin2}
If $k\geq 4$ is even and $n\geq 1$ is odd, then $X_{g,k}$ has no spin structure. 
\end{lem}

\begin{proof}
First, we prove that $X_{g,k}-\nu (\Sigma _g)$ admits a spin structure. 
Since $\{ \gamma _i^l\mid 1\leq i\leq 2n, 1\leq l\leq k-1\}$ is a basis of $H_1(\Sigma _g; \Z _2)$, a quadratic form on $H_1(\Sigma _g; \Z _2)$ is determined by the image of $\gamma _i^l$. 
We denote a quadratic form $q \colon H_1(\Sigma _g;\Z _2)\to \Z _2$ by $q(\gamma _i^l)=1$ for any $1\leq i\leq 2n$ and $1\leq l\leq k-1$. 
By Proposition~\ref{prop_Stipsicz1}, it is enough for proving that $X_{g,k}-\nu (\Sigma _g)$ admits a spin structure to show that $q(\gamma _i^k)=1$ for $1\leq i\leq 2n$, $q(\gamma _{2n+1}^l)=1$ for $1\leq l\leq k$, and $q(\alpha _i^l)=1$ for $1\leq i\leq n$ and $1\leq l\leq k-1$.  

As a $\Z _2$-homology class, we have $\gamma _i^k=\gamma _i^1+\cdots +\gamma _i^{k-1}$ for $1\leq i\leq 2n$. 
Since $\gamma _i^l\cdot \gamma _i^{l^\prime }=0\in \Z _2$ for $1\leq l, l^\prime \leq k-1$, we have
\begin{eqnarray*}
q(\gamma _i^k)&=&q(\gamma _i^1+\cdots +\gamma _i^{k-1})=q(\gamma _i^1)+\cdots +q(\gamma _i^{k-1})\\
&=&k-1\\
&=&1\quad \text{in }\Z_2.
\end{eqnarray*}
Similarly, as a $\Z _2$-homology class, we have $\gamma _{2n+1}^l=\gamma _1^l+\gamma _3^l+\cdots +\gamma _{2n-1}^{l}$ for $1\leq l\leq k$. 
Since $\gamma _{2i-1}^l\cdot \gamma _{2i^\prime -1}^{l}=0\in \Z _2$ for $1\leq i, i^\prime \leq n+1$, we have 
\begin{eqnarray*}
q(\gamma _{2n+1}^l)&=&q(\gamma _1^l+\gamma _3^l+\cdots +\gamma _{2n-1}^{l})=q(\gamma _1^l)+q(\gamma _3^l)+\cdots +q(\gamma _{2n-1}^{l})\\
&=&n\\
&=&1\quad \text{in }\Z_2.
\end{eqnarray*}
As a $\Z _2$-homology class, we also have $\alpha _i^l=\gamma _{2i}^l+\gamma _{2i-1}^{l+1}$ for $1\leq i\leq n$ and $1\leq l\leq k-1$. 
Since $\gamma _{2i}^l\cdot \gamma _{2i-1}^{l+1}=1\in \Z _2$, we have 
\begin{eqnarray*}
q(\alpha _i^l)&=&q(\gamma _{2i}^l+\gamma _{2i-1}^{l+1})=q(\gamma _{2i}^l)+q(\gamma _{2i-1}^{l+1})+\gamma _{2i}^l\cdot \gamma _{2i-1}^{l+1}\\
&=&1+1+1\\
&=&1\quad \text{in }\Z_2.
\end{eqnarray*}
Thus, $q \colon H_1(\Sigma _g;\Z _2)\to \Z _2$ is a quadratic form such that $q(\gamma )=1\in \Z _2$ for each vanishing cycle $\gamma $ of $f_{g,k} \colon X_{g,k} \to S^2$, and by Proposition~\ref{prop_Stipsicz1}, $X_{g,k}-\nu (\Sigma _g)$ admits a spin structure. 

By this fact and Proposition~\ref{prop_Stipsicz2}, $X_{g,k}$ admits a spin structure if and only if for some dual $\gamma \in H_2(X;\Z )$ of $\Sigma _g$, we have $\gamma \cdot \gamma \equiv 0$ mod $2$. 
By Lemma~\ref{lem_section}, the Lefschetz fibration $f_{g,k} \colon X_{g,k} \to S^2$ has a $(-1)$-section $s\colon S^2\to X_{g,k}$. 
Since $s(S^2)$ transversely intersects with the regular fiber $\Sigma _g$ at one point, the homology class $[s(S^2)]\in H_2(X_{g,k};\Z )$ is a dual of $\Sigma _g$ with $[s(S^2)]\cdot [s(S^2)]=-1\equiv 1\not \equiv 0$ mod $2$. 
The mod $2$ self-intersection number of a dual of $\Sigma _g$ does not depend on a choice of a dual since $X_{g,k}-\nu (\Sigma _g)$ admits a spin structure. 
Thus, the mod $2$ self-intersection number of any dual of $\Sigma _g$ is 1. 
Therefore, by Proposition~\ref{prop_Stipsicz2}, $X_{g,k}$ does not admit a spin structure and we have completed the proof of Lemma~\ref{lem_nonspin2}. 

\end{proof}

\begin{rem}
When $n=1$, namely $g=k-1\geq 2$, the Lefschetz fibration $f_{g,k}\colon X_{g,k}\to S^2$ corresponds to the positive relator $(\widetilde{h}_1\widetilde{t}_{3}\widetilde{a}_1)^k$.
Then we remark that $k=g+1$ and 
\begin{align*}
\widetilde{h}_{1}&=t_{\gamma _{1}^1}t_{\gamma _{2}^{1}}t_{\gamma _{1}^2}t_{\gamma _{2}^{2}}t_{\gamma _{1}^3}\cdots t_{\gamma _{2}^{k-1}}t_{\gamma _{1}^{k}}, \\
\widetilde{t}_{3}&=\widetilde{t}_{1}=t_{\gamma _1^1}t_{\gamma _1^2}\cdots t_{\gamma _1^k},\\
\widetilde{a}_1&=t_{\alpha _1^1}t_{\alpha _1^2}\cdots t_{\alpha _1^{k-1}}.
\end{align*}
Since $t_{\gamma _1^{l+1}}(\alpha _1^l)=\gamma _2^l$ for $1\leq l\leq k-1$, by elementary transformations, we have 
\begin{eqnarray*}
&&(\widetilde{h}_1\widetilde{t}_{3}\widetilde{a}_1)^{g+1}\\
&=&\Bigl( (t_{\gamma _{1}^1}t_{\gamma _{2}^{1}}t_{\gamma _{1}^2}t_{\gamma _{2}^{2}}t_{\gamma _{1}^3}\cdots t_{\gamma _{2}^{k-1}}t_{\gamma _{1}^{k}})(t_{\gamma _1^1}t_{\gamma _1^2}\underset{\rightarrow }{\underline{t_{\gamma _1^3}}}\cdots \underset{\rightarrow }{\underline{t_{\gamma _1^k}}}\cdot t_{\alpha _1^1}t_{\alpha _1^2}\cdots t_{\alpha _1^{k-1}})\Bigr) ^{g+1}\\
&\sim &\left( (t_{\gamma _{1}^1}t_{\gamma _{2}^{1}}t_{\gamma _{1}^2}t_{\gamma _{2}^{2}}t_{\gamma _{1}^3}\cdots t_{\gamma _{2}^{k-1}}t_{\gamma _{1}^{k}})(t_{\gamma _1^1}\underline{t_{\gamma _1^2}t_{\alpha _1^1}}\cdot \underline{t_{\gamma _1^3}t_{\alpha _1^2}}\cdots \underline{t_{\gamma _1^k}t_{\alpha _1^{k-1}}})\right) ^{g+1}\\
&\sim &\Bigl( (t_{\gamma _{1}^1}t_{\gamma _{2}^{1}}t_{\gamma _{1}^2}t_{\gamma _{2}^{2}}t_{\gamma _{1}^3}\cdots t_{\gamma _{2}^{k-1}}t_{\gamma _{1}^{k}})(t_{\gamma _1^1}t_{t_{\gamma _1^2}(\alpha _1^1)}t_{\gamma _1^2}\cdot t_{t_{\gamma _1^3}(\alpha _1^2)}t_{\gamma _1^3}\cdots t_{t_{\gamma _1^k}(\alpha _1^{k-1})}t_{\gamma _1^k})\Bigr) ^{g+1}\\
&=&\Bigl( (t_{\gamma _{1}^1}t_{\gamma _{2}^{1}}t_{\gamma _{1}^2}t_{\gamma _{2}^{2}}t_{\gamma _{1}^3}\cdots t_{\gamma _{2}^{k-1}}t_{\gamma _{1}^{k}})(t_{\gamma _1^1}t_{\gamma _2^1}t_{\gamma _1^2}t_{\gamma _2^2}t_{\gamma _1^3}\cdots t_{\gamma _2^{k-1}}t_{\gamma _1^k})\Bigr) ^{g+1}\\
&=&(t_{\gamma _1^1}t_{\gamma _2^1}t_{\gamma _1^2}t_{\gamma _2^2}t_{\gamma _1^3}\cdots t_{\gamma _2^{k-1}}t_{\gamma _1^k})^{2g+2}.
\end{eqnarray*}
Hence the Lefschetz fibration $f_{g,k}\colon X_{g,k}\to S^2$ for $n=1$ and $k\geq 3$ with $g=n(k-1)$ is isomorphic to a Lefschetz fibration of genus $g$ which is corresponding to a $(2g+1)$-chain relation. %$(t_{\gamma _1^1}t_{\gamma _2^1}t_{\gamma _1^2}t_{\gamma _2^2}t_{\gamma _1^3}\cdots t_{\gamma _2^{k-1}}t_{\gamma _1^k})^{2g+2}=1$. 
Then, by Proposition~3.10 in \cite{Endo-Nagami}, the signature of $X_{g,k}$ is equal to $-2g(g+2)$, and by Lemma~10 in \cite{Akhmedov-Monden}, $f_{g,k}\colon X_{g,k}\to S^2$ is obtained as a double branched covering of $\mathbb{CP}^2\sharp \overline{\mathbb{CP}^2}$ which is branched along some smooth algebraic curve. 
\end{rem}

%\newpage
\par
{\bf Acknowledgement:} The author would like to express his gratitude to Hisaaki Endo, for his encouragement and helpful advices. 
The authors also wish to thank Takahiro Oba for his comments and helpful advices. 
The author was supported by JSPS KAKENHI Grant Numbers JP19K23409 and JP21K13794.
%The authors also wish to thank Susumu Hirose for his comments and helpful advices.
%JST CREST Grant Number JPMJCR17J4, Japan. 
%Grant-in-Aid for Research Activity Start-up


\begin{thebibliography}{99}

\bibitem{Akhmedov-Monden}
A. Akhmedov, N. Monden, \emph{Constructing Lefschetz fibrations via daisy substitutions}, Kyoto J. Math. \textbf{56} (2016), no. 3, 501--529. 

\bibitem{Birman-Hilden1}
J. S. Birman, H. M. Hilden, \emph{On the mapping class groups of closed surfaces as covering spaces},
Advances in the Theory of Riemann Surfaces (Proc. Conf., Stony Brook, N.Y., 1969), Ann. of Math. Studies 66, Princeton Univ. Press, Princeton, N.J., 81--115, 1971.

\bibitem{Birman-Hilden2}
J. S. Birman, H. M. Hilden, \emph{Lifting and projecting homeomorphisms}, Arch. Math. (Basel) \textbf{23} (1972), 428--434.

\bibitem{Birman-Hilden3}
J. S. Birman, H. M. Hilden, \emph{On isotopies of homeomorphisms of Riemann surfaces}, Ann. of Math. (2) \textbf{97} (1973), 424--439.

\bibitem{Dehn}
M. Dehn, \emph{Die Gruppe der Abbildungsklassen}, Acta Math. \textbf{69} (1938), 135--206.


\bibitem{Endo-Nagami}
H. Endo, S. Nagami, \emph{Signature of relations in mapping class groups and non-holomorphic Lefschetz fibrations}, Trans. Amer. Math. Soc. \textbf{357} (2005), 3179--3199.

\bibitem{Farb-Margalit}
B. Farb, D. Margalit, \emph{A primer on mapping class groups}, Princeton University Press, Princeton, NJ, 2012.

\bibitem{Gompf-Stipsicz}
R. E. Gompf, A. I. Stipsicz, \emph{4-manifolds and Kirby calculus}, Graduate Studies in Mathematics \textbf{20}, American Mathematical Society, Providence, RI, 1999. 

\bibitem{Gurtas1}
Y. Z. Gurtas, \emph{Positive Dehn twist expressions for some new involutions in the mapping class group II}, arXiv:math/0404311.

\bibitem{Gurtas2}
Y. Z. Gurtas, \emph{Positive Dehn twist expressions for some elements of finite order in the mapping class group}, arXiv:math/0501385.

\bibitem{Gurtas3}
Y. Z. Gurtas, \emph{Positive Dehn Twist Expression for a $\Z _3$ action on $\Sigma _g$}, arXiv:0808.0752. 

  
\bibitem{Ghaswala-Winarski1}
T. Ghaswala, R. R. Winarski, \emph{The liftable mapping class group of balanced superelliptic covers}, New York J. Math. \textbf{23} (2017), 133--164. 

\bibitem{Ghaswala-Winarski2}
T. Ghaswala, R. R. Winarski, \emph{Lifting homeomorphisms and cyclic branched covers of spheres}, Michigan Math. J. \textbf{66} (2017), no. 4, 885--890. 

\bibitem{Hirose}
S. Hirose, \emph{Presentations of periodic maps on oriented closed surfaces of genera up to 4}, 
Osaka J. of Mathematics \textbf{47} (2010), 385--421.

\bibitem{Hirose-Omori}
S. Hirose, G. Omori, \emph{Finite presentations for the balanced superelliptic mapping class groups}, arXiv:2203.13413. 


\bibitem{Ishizaka}
M. Ishizaka, \emph{Presentation of hyperelliptic periodic monodromies and splitting families}, Rev. Mat. Complut. \textbf{20} (2007), 483--495.

\bibitem{Kas}
A. Kas, \emph{On the handlebody decomposition associated to a Lefschetz fibration}, Pacific J. Math. \textbf{89} (1980), no. 1, 89--104. 

\bibitem{Korkmaz}
M. Korkmaz, \emph{Noncomplex smooth 4-manifolds with Lefschetz fibrations}, Internat. Math. Res. Notices, (2001), 115--128.


\bibitem{Matsumoto}
Y. Matsumoto, \emph{Lefschetz fibrations of genus two --- a topological approach}, 
in Topology and Teichmüller Spaces (Katinkulta, 1995), World Sci. Publ., River Edge, NJ., 123--148, 1996.

\bibitem{Stipsicz}
A. I. Stipsicz, \emph{Spin structures on Lefschetz fibrations}, Bull. London Math. Soc. \textbf{33} (2001), no. 4, 466--472. 

\bibitem{Wajnryb}
B. Wajnryb, \emph{Mapping class group of a surface is generated by two elements}, Topology \textbf{35} (1996), no. 2, 377--383.




\if0

\bibitem{Birman}
J. S. Birman, \emph{Mapping class groups and their relationship to braid groups}, Comm. Pure Appl. Math. \textbf{22} (1969), 213--238.

\bibitem{Birman-Chillingworth} 
J. S. Birman, D. R. J. Chillingworth, \emph{On the homeotopy group of a non-orientable surface}, Proc. Camb. Philos. Soc. \textbf{71} (1972), 437--448.

\bibitem{BEMS}
E. Bujalance, J. J. Etayo, E. Mart\'{i}nez, B. Szepietowski, \emph{On the connectedness of the branch loci of non-orientable unbordered Klein surfaces of low genus}, Glasg. Math. J. \textbf{57} (2015), 211--230.

\bibitem{Dugger}
D. Dugger, \emph{Involutions on surfaces}, J. Homotopy Relat. Struct. \textbf{14} (2019), 919--992.

\bibitem{Endo-Nagami}
H. Endo, S. Nagami, \emph{Signature of relations in mapping class groups and non-holomorphic Lefschetz fibrations}, Trans. Amer. Math. Soc. \textbf{357} (2005), 3179--3199.

\bibitem{Evans-Kolbe}
M. E. Evans, B. Kolbe, \emph{Isotopic tiling theory for hyperbolic surfaces}, Geom. Dedicata \textbf{212} (2021), 177--204.

\bibitem{Hirose}
S. Hirose, \emph{Presentations of periodic maps on oriented closed surfaces of genera up to 4}, 
Osaka J. of Mathematics \textbf{47} (2010), 385--421.

\bibitem{Hirose2}
S. Hirose, \emph{Generators for the mapping class group of a nonorientable surface}, Kodai Math. J. \textbf{41} (2018), 154--159.

\bibitem{Ishizaka}
M. Ishizaka, \emph{Presentation of hyperelliptic periodic monodromies and splitting families}, Rev. Mat. Complut. \textbf{20} (2007), 483--495.

\bibitem{Korkmaz1}
M. Korkmaz, \emph{Noncomplex smooth 4-manifolds with Lefschetz fibrations}, Internat. Math. Res. Notices, (2001), 115--128.

\bibitem{Korkmaz2}
M. Korkmaz, \emph{Mapping class groups of nonorientable surfaces}, Geom. Dedicata. \textbf{89} (2002), 109--133.

\bibitem{Lickorish1} 
W. B. R. Lickorish, \emph{Homeomorphisms of non-orientable two-manifolds}, Proc. Camb. Philos. Soc. \textbf{59} (1963), 307--317.

\bibitem{Macbeath} 
A. M. Macbeath, \emph{The classification of non-euclidean plane crystallographic groups}, Canadian J. Math. \textbf{19} (1967), 1192--1205.

\bibitem{Matsumoto}
Y. Matsumoto, \emph{Lefschetz fibrations of genus two --- a topological approach}, 
in Topology and Teichmüller Spaces (Katinkulta, 1995), World Sci. Publ., River Edge, NJ., 123–148, 1996.

\bibitem{Mumford}
D. Mumford, \emph{Abelian quotients of the Teichm\"uller modular group}, J. Analyse Math., \textbf{18} (1967), 227--244.

\bibitem{Murasugi}
K. Murasugi, \emph{The center of a group with a single defining relation}, Math. Ann. \textbf{155} (1964), 246--251. 

\bibitem{Nielsen}
J. Nielsen, \emph{Die Struktur periodischer Transformationen von Fl\''{a}chen}, Math. -fys. Medd. Danske Vid. Selsk. \textbf{15} (1937) (English transl. in ``Jakob Nielsen collected works, Vol. 2'', 65--102).

\bibitem{Nielsen2}
J. Nielsen, \emph{Abbildungsklassen endlicher Ordnung}, Acta Math. \textbf{75} (1943), 23--115.


\bibitem{Stukow2}
M. Stukow, \emph{A finite presentation for the hyperelliptic mapping class group of a nonorientable surface}, Osaka J. Math. \textbf{52} (2015), 495--514.

\bibitem{Szepietowski1} 
B. Szepietowski, \emph{Crosscap slides and the level 2 mapping class group of a nonorientable surface}, Geom. Dedicata \textbf{160} (2012), 169--183.

\bibitem{Szepietowski2} 
B. Szepietowski, \emph{A finite generating set for the level 2 mapping class group of a nonorientable surface}, Kodai Math. J. \textbf{36} (2013), 1--14.

\fi

\end{thebibliography}
\end{document}